\documentclass[12pt]{amsart}
\usepackage{a4}
\usepackage{amsmath}
\usepackage{amssymb}
\usepackage{graphicx}
\def\+{\oplus}

\newcommand{\R}{{\mathbb R}}

\newcommand{\cI}{{\mathcal I}}

\newcommand{\cB}{{\mathcal B}}

\newcommand{\cL}{{\mathcal L}}

\newcommand{\cC}{{\mathcal C}}

\newcommand{\cV}{{\mathcal V}}

\newcommand{\cR}{{\mathcal R}}

\newcommand{\td}{\tilde}

\renewcommand{\phi}{\varphi}
\renewcommand{\a}{{\alpha}}

\newcommand{\g}{{\gamma}}
\renewcommand{\d}{{\delta}}
\newcommand{\e}{\epsilon}

\newcommand{\G}{\Gamma}
\renewcommand{\O}{\Omega}

\newcommand{\pd}{\partial}
\newcommand{\im}{\mathfrak{i}}
\newcommand{\Inc}{\hbox{Inc}}

\def\squareforqed{\hbox{\rlap{$\sqcap$}$\sqcup$}}
\def\qed{\ifmmode\else\unskip\quad\fi\squareforqed}
\def\smartqed{\def\qed{\ifmmode\squareforqed\else{\unskip\nobreak\hfil
\penalty50\hskip1em\null\nobreak\hfil\squareforqed
\parfillskip=0pt\finalhyphendemerits=0\endgraf}\fi}}

\newtheorem{remark}{\textbf{Remark}}[section]
\newtheorem{lemma}{\textbf{Lemma}}[section]
\newtheorem{theorem}{\textbf{Theorem}}[section]

\newtheorem{proposition}{\textbf{Proposition}}[section]

\newtheorem{definition}{\textbf{Definition}}[section]

\newtheorem{example}{\textbf{Example}}[section]
\numberwithin{equation}{section}
\begin{document}
\title[Eikonal equation on LEP space]{Eikonal equations on
ramified spaces}

\author{Fabio Camilli}
\address{Dipartimento di Scienze di Base e Applicate per l'Ingegneria,  ``Sapienza" Universit{\`a}  di Roma,
 00161 Roma, Italy, (e-mail:camilli@dmmm.uniroma1.it)}
\author{Dirk Schieborn}
\address{Eberhard-Karls University,  T\"ubingen, Germany  (e-mail:Dirk@schieborn.de)}
\author{Claudio Marchi}
\address{Dip. di Matematica, Universit\`a di Padova, via Trieste 63, 35121 Padova, Italy (email: marchi@math.unipd.it).}
\date{\today}

\begin{abstract}
We generalize the results  in \cite{sc} to
higher dimensional ramified spaces. For this purpose we introduce ramified manifolds
and, as special cases, locally elementary polygonal ramified spaces (LEP spaces). On LEP
spaces we develop a theory of viscosity solutions for  Hamilton-Jacobi equations, providing
existence and uniqueness results.
\end{abstract}
%
%
\maketitle
\section{Introduction}
In \cite{s}, \cite{sc}  a theory of viscosity solution  for Hamilton-Jacobi equations of
eikonal type  on topological networks was developed  providing existence, uniqueness and stability results.
In this  paper  we generalize these results   to
higher dimensional {\it ramified spaces}. \par
In literature, many different ways of introducing  ramified spaces  (cf.
\cite{lu}, \cite{ni1}, \cite{ni2}) or branched manifolds (cf. \cite{wi}) are available.
The definitions vary in different aspects, depending on the kind of theory to be developed.
In a general approach, subsets of classic differentiable manifolds are glued together
along parts of their boundaries by means of the topological gluing operation. Another,
more specific, definition requires the uniqueness of the  tangent space  at ramification
points (cf. \cite{wi}) by describing how the branches should be situated relatively to each
other in the ambient space.\par
 Here we choose an approach which is very similar to the
concept of a manifold with boundary. The basic idea is that, in contrast to classic topological manifolds, besides  points at which
it is locally homeomorphic to an Euclidean space ({\it simple} points), a ramified topological manifold may also contain
{\it ramification} points at which it is locally homeomorphic to some kind of ``Euclidean ramified
space''. The latter,  called \emph{elementary ramified space}, can be visualized as a collection of closed Euclidean half spaces glued together at their boundary hyperplanes.
Consequently, small neighborhoods of a given ramification point split up into different
branches corresponding to the branches of the homeomorphic elementary ramified space.
If we endow these ramified topological manifolds with suitable differentiable structures, then
we end up with an extension of the concept of tangent space at ramification points. This
generalization should have the property that a real function defined in a neighborhood
of a ramification point can be differentiated in the direction of each branch (of course, incident at the ramification point). In other words, each branch contributes a different tangent space.\par
Once we have introduced the differentiable structure on ramified spaces, we will see that for each of the
branches emanating from a fixed ramification point $x$, a normal direction at $x$ on
this branch  is well-defined. The possibility to differentiate in the normal
directions at ramification points is crucial for our theory, as it will turn out that a general
definition of viscosity solutions on ramified manifolds depends on this very possibility.
In fact, the notion of viscosity solutions introduced in \cite{s}, \cite{sc} differs from its classical origin by
the transition conditions we have additionally imposed at ramification points. The concept of $(j, k)$-test functions (see definition~\ref{testfunction}) allows us to  ignore  the ramification set by treating two branches as a single connected one; indeed, the  $(j, k)$-differentiability links the two derivatives of a function with
respect to a given pair of branches which are incident at the same ramification point. It suggests itself
to apply this pattern in case of manifolds of dimension $n$ which have a certain manifold
of dimension $n - 1$ in common, as long as this manifold is smooth enough to ensure that
we have well-defined normal derivatives with respect to each incident  branch manifold.\par
In order not to get lost in too general approaches, we restrict ourselves to a rather simple, but still sufficiently general,
kind of ramified manifolds, the so-called \emph{locally elementary polygonal ramified spaces} (briefly, LEP
spaces), which are characterized by two main criteria: on the one hand, LEP spaces
are ramified spaces in the sense of Lumer \cite{lu} (see definition \ref{ramifiedspace}) meeting the additional
requirement that each branch is a flat $n$-dimensional submanifold of $\R^{n+1}$. On the other
hand, they are ramified manifolds in the sense described above. Hence they can be
visualized as polygonal subsets of hyperplanes in $\R^{n+1}$ which are glued together along
certain edges, with the restriction that  corner points cannot occur. The term  ``locally
elementary''  refers to the fact  that they are locally homeomorphic to an
open subset either of a $n$-dimensional Euclidean space or of an elementary ramified space.
Once the notion of viscosity solutions has been correctly extended to LEP
spaces, the development of the theory follows the line  of the one devised for topological networks in \cite{sc}. Consequently
we prove  a comparison principle giving the uniqueness of the continuous viscosity solution. Moreover we show existence
of the viscosity solution via an adaptation of the Perron's method and we also provide a representation formula
for the solution of the Dirichlet problem.\par
We mention that Hamilton-Jacobi equation and viscosity solutions on differentiable manifolds have been studied
in \cite{afl},  \cite{mm}.  The theory of linear and semilinear differential equations  on nonsmooth manifolds such as  ramified spaces  has been developed, since the seminal paper   \cite{lu}, in a large extent \cite{lls} and it is currently an active field of research
(\cite{ll}, \cite{kl}). For fully nonlinear equations such as Hamilton-Jacobi equations, the theory is at the beginning and, besides \cite{sc} and the companion paper \cite{cfs},    different approaches have been  pursued for the case of networks in the recent papers \cite{acct} and \cite{imz} and for
 stratified domains in \cite{bh} . The present  paper can been seen as a first  attempt to extend the theory of viscosity solutions to general ramified spaces.\par
The paper is organized as follows. In section \ref{section1} we introduce the definition and give various examples of ramified spaces. In section \ref{section2}  we study the differential structure
of a  ramified space. Section \ref{section3} is devoted to the notion of viscosity solution, while in section \ref{section4} and \ref{section5}
we prove uniqueness and, respectively, existence  of a viscosity solution. In section \ref{section6} we consider the Dirichlet problem
and we obtain a representation formula for its solution.
\section{Ramified spaces}\label{section1}
In this section we introduce the geometric objects we will study in this paper.
The general definition of ramified space is due to Lumer \cite{lu}.
\begin{definition}\label{ramifiedspace}
Let $R^\star$ be a non-empty, locally compact space with a countable basis. Let $\cL=\{R_j\}_{j\in J}$  be a countable family of non empty open subsets $R_j$ of $R^\star$ and let $N^\star_E$ be a closed, possible empty, subset of $N^\star:=R^\star\setminus \cup_{j\in J}R_j$ with the property that it contains each point of $N^\star$ which is contained in the boundary  of exactly one $R_j$. Then $R:=R^\star\setminus N_E^\star$ is a ramified space (induced by $(R^\star,\cL,N^\star_E)$) if
\begin{itemize}
\item $\bar R_j\cap\bar R_k\subset \pd\bar R_j\cap\pd\bar R_k$ for all $j\neq k$, $j,k\in J$,
\item $R^\star=\cup_{j\in J}\bar R_j$,
\item $\{R_j\}_{j\in J}$   is locally finite in $R^\star$,
\item $R$ is connected.
\end{itemize}
The set $\pd R=N_E^\star$ is called the boundary of $R$ while the set $N_R=N^\star\setminus N^\star_E$  the ramification space of $R$. We set $\partial_R R_j:=\partial R_j \cap N_R$ and $\tilde R_j:=R_j\cup \partial_R R_j$.
\end{definition}
We also consider polygonal ramified space.
\begin{definition}\label{polramifiedspace}
A ramified space $R$ is said a  $n$-dimensional polygonal ramified space if
\begin{itemize}
  \item $R^\star\subset \R^{n+1}$ with the endowed topology,
  \item  For each $j\in J$, there is a hyperplane $P_j\subset\R^{n+1}$ such that $R_j$ is a bounded subset of $P_j$,
  \item  All $P_j$, $j\in J$, are pairwise distinct.
\end{itemize}
\end{definition}
We give some examples of ramified spaces and polygonal ramified space.
\begin{example}\label{TN}
A topological network   is a collection of pairwise different points in $\R^n$ connected by continuous, non self-intersecting curves.
More precisely (see \cite{sc}), let $V=\{v_i,\,i\in I\}$ be a finite collection of pairwise different points in $\R^n$ and
let $\{\pi_j,\,j\in J\}$ be a finite collection of continuous, non self-intersecting curves in $\R^n$ given by
$\pi_j:[0,l_j]\to\R^n,\, l_j>0,\,j\in J$. Defined
 $e_j:=\pi_j((0,l_j))$, $\bar e_j:=\pi_j([0,l_j])$  and $E:=\{e_j:\, j\in J\}$,   assume that
\begin{itemize}
  \item[i)] $\pi_j(0), \pi_j(l_j)\in V$ for all $j\in J$,
  \item[ii)]$\#(\bar e_j\cap V)=2$ for all $j\in J$,
  \item[iii)] $\bar e_j\cap \bar e_k\subset  V$, and $\#(\bar e_j\cap \bar e_k)\le 1$ for all $j,k \in J$, $j\neq k$.
  \item[iv)] For all $v, w \in V$ there is a path with end-points  $v $ and $w$ (i.e.  a sequence of edges $\{e_j\}_{j=1}^N$ such that
  $\#(\bar e_j\cap \bar e_{j+1})=1$ and  $v\in \bar e_1$, $w\in \bar e_N$).
\end{itemize}
Then $\bar \G:=\bigcup_{j\in J}\bar e_j\subset \R^n$ is called a (finite)  \emph{topological network} in $\R^n$.\\
A topological network is a ramified space with $R^\star=V\cup E$,  $R_j=e_j$, $N^\star=V$, $N^\star_E$ any subset of $V$ containing
all the vertices with only one incident edge and $N_R=N^\star\setminus N^\star_E$.
\end{example}
%
\begin{example}\label{nTN}
Let $\G$ be a topological network as in Example~\ref{TN}. For $v\in V$, set $deg(v)$ the number of the arc $e_j$, $j\in J$, incident at the vertex $v$ and
define $\tilde \G=\G\setminus\{v\in V:\, deg(v)=1\}$. Then for $n\ge 2$, the set $M:= \tilde \G\times \R^{n-1}$ is called a \emph{$n$-dimensional
topological network} (\cite{ni2}). In this case $R^\star=(V\cup E)\times \R^{n-1}$, $R_j=e_j\times \R^{n-1}$, $N^\star=V\times \R^{n-1}$,
$N^\star_E$ any subset of $N^\star$ containing the set $\cup_{\{v: deg(v)=1\}}(v\times \R^{n-1})$.
\end{example}
If the edges $\{e_j\}_{j\in J}$ of a topological network $\G$ are segments, then $\G$  and the corresponding $n$-dimensional
topological networks defined as in example \ref{nTN} are  polygonal  ramified spaces in the sense of definition \ref{polramifiedspace}. See figure~1.
\begin{example}\label{cube}
Let $\O^\star$ be the surface of the $(n+1)$-dimensional cube $C^{n+1}\subset\R^{n+1}$ and let $\O_j$, $j=1,\dots, 2(n+1)$, be its open faces.
Furthermore let $\partial \O:=N^\star_E$ be any closed (possible empty) subset of the union of the edges of the cube with the property that $\O=\O^\star\setminus N^\star_E$ is connected. Then $\O:=\O^\star\setminus N_E^\star$ is a polygonal ramified space.
\end{example}
An important example of ramified space is the \emph{elementary ramified space}, since it is the space  of  the  parameters  for   LEP spaces
and ramified manifolds we will define in the following.
\begin{definition}\label{ERS}
Given $n\ge 1$ and $r\ge 2$, a  $n$-dimensional elementary ramified space of order  $r$, denoted by $\cR^n_r$, is the union of $r$
half spaces $\cR^n_{r,j}$, $j\in\{1,\dots,r\}$, of dimension $n$ which are included in $\R^{n+1}$ and have $\R^{n-1}$ in common.\\
If we set   $\R^n_{\ge 0}=\{(x_1,x')\in\R\times\R^{n-1}:\,x_1\ge 0\}$, then we can identify $\cR^n_r$ and $\cR^n_{r,j}$ with
\begin{align*}
 &\cR^n_r=(\R^n_{\ge 0}\times\{1,\dots,r\})/\Re\\
 &\cR^n_{r,j}= \{(x,j):\, x\in \R^n_{\ge 0}\}
\end{align*}
where $\Re$ is the equivalence relation which for each choice of $x'\in\R^{n-1}$  identifies the points $((0,x'),j)\in \cR^n_j$ for $j\in\{1,\dots,r\}$. The set $\cR^n_{r,j}$ is said  the (closed) \emph{$j$-branch} of $\cR^n_r$ while the set
\[ \Sigma^n_r=\{((0,x'),j):\, x'\in\R^{n-1},\, j\in\{1,\dots,r\}\}
 \]
is called the \emph{ramification space}  of $\cR^n_r$.\par
Endowed with the topology induced by the path distance, $\cR^n_r$ is a connected, separable, locally compact topological
space. Observe that  $\cR^n_r$ can be identified with $\R^n$ if $r=2$.
\end{definition}
In order to give the definition of ramified manifolds, we need to introduce the notion of diffeomorphism on $\cR^n_r$.
\begin{definition}\label{diffeo}
\emph{1)} Let $U\subset\cR^n_r$ be an open set and $f:U\to \R^m$.  Then, for $1\le l\le \infty$,  $f$ is said $C^l$-differentiable at $x\in U$
    if the following holds:
    \begin{itemize}
    \item[i)] If $x\in\cR^n_{r,j}\setminus \Sigma^n_r$, for some $j\in \{1,\dots,r\}$, then $f$ is $l$ times continuously differentiable at $x$ in the standard sense.
      \item[ii)] If  $x\in\Sigma^n_r$, then for each $j\in\{1,\dots,r\}$, there is a domain  $V_j\subset\R^n$ and $f_j\in C^l(V_j,\R^m)$ such that
      $x\in V_j$ and $f_j\equiv f$ on $V_j\cap \cR^n_{r,j}$ (having identified $\cR^n_{r,j}$ with $\R^n_{\ge 0}$).
    \end{itemize}
\emph{2)} Let $U,V\subset\cR^n_r$ be open sets and $\phi:U\to V$ an homeomorphism. Then $\phi$ is said a diffeomorphism if for all $j\in\{1,\dots,r\}$ the respective restrictions of  $\phi$ and $\phi^{-1}$ to $\cR^n_{r,j}\cap U$ and to $\cR^n_{r,j}\cap V$ are $C^\infty$-differentiable.
\end{definition}
We are now ready to give the definition of topological ramified manifold and differentiable ramified manifold.
\begin{definition}
A set $M$ is called a $n$-dimensional topological ramified manifold if it is endowed with a Hausdorff topology and if  for any $x\in M$, there is
a neighborhood $U\subset M$ of $x$ such that there is an integer $r=r(x)\ge 2$, an open set $V\subset \cR^n_r$ with $V\cap \Sigma^n_r\neq \emptyset$
and a homeomorphism  $X:U\to V$ with $X(x)\in\Sigma^n_r$.
\end{definition}
The number $r(x)$ is called ramification order of $x$. A point $x\in M$ is said a simple point if $r(x)=2$, a ramification point if $r(x)\ge 3$.
The set of all ramification points is denoted by $\Sigma$ and it is called ramification space of $M$. If $x\in\Sigma$, we set
$\Inc_x:=\{j\in J:\,x\in\pd \O_j\}$.
\begin{remark}
Observe that, since $\cR^n_2$ can be identified with $\R^n$,    topological ramified manifolds are locally homeomorphic to a $n$-dimensional Euclidean space at simple points.
\begin{definition}
A set $M$ is called a $n$-dimensional differentiable ramified manifold if
$M$ is a $n$-dimensional topological ramified manifold
and there is a family of local charts $\{U_\a,X_\a\}$, i.e. open set $U_\a\subset M$
and injective mappings $X_\a:U_\a\to\cR^n_{r(\a)}$, with the following properties
\begin{itemize}
  \item[i)] For any $\a,\a'$ with $V=U_\a\cap U_{\a'}\neq \emptyset$, the sets $X_\a(V)$ and $X_{\a'}(V)$
  are open in $\cR^n_{r(\a)}$ and $\cR^n_{r(\a')}$, respectively. Moreover the map $\phi:X_\a(V)\to X_{\a'}(V)$
  given by $\phi:=X_{\a'}\circ X_\a^{-1}$ is a diffeomorphism in the sense of definition \ref{diffeo}.
  \item[ii)] $\bigcup_\a U_\a=M$.
  \item[iii)] The family $\{U_\a,X_\a\}$ is maximal with respect to the conditions i) and ii).
\end{itemize}
\end{definition}
\end{remark}

We introduce a class of flat ramified manifolds.
\begin{definition}\label{LEP}
A $n$-dimensional polygonal ramified space $\O$ (see definition \ref{polramifiedspace}) is called locally elementary if it is also a differentiable
manifold. Locally elementary ramified space will be called LEP spaces in the following.
\end{definition}

\begin{example}
Topological networks and $n$-dimensional topological networks are  topological ramified manifolds.
If the maps $\{\pi_j\}_{j\in J}$ in the definition of $\G$ are diffeomorphisms, they are also   differentiable ramified manifolds.
If  the edges $\{e_j\}_{j\in J}$ of $\G$ are segments, a $n$-dimensional topological network  is a   LEP space.\par
The set of ramification points $\Sigma$ is given by $\{v\in V:\,deg(v)>1\}$ for a topological network and by  $\cup_{r>2} M_r$ where $M_r=\cup_{\{v\in V:\,deg(v)=r\}}(\{v\}\times \R^{n-1})$ for a $n$-dimensional topological network. Note that $\Sigma=N_R$ for a LEP space.
\end{example}
\setlength{\unitlength}{1cm}
\begin{picture}(14,9.2)
 \put(0,0){\framebox(12,9.2)}
\put(0,8.8){ {\bf Figure 1}}
\linethickness{0.8mm}
\multiput(4,.9)(0,.4){15} {\line(0,1){0.2}}
\multiput(6,1.55)(0,.4){15} {\line(0,1){0.2}}
\linethickness{0.2mm}
\put(4,3.8){\line(-1,0){3.5}}
\put(4,3.8){\line(3,1){2}}
\put(4,3.8){\line(-1,1){2}}
\put(4,3.8){\circle*{.5}}
\put(4,3.81){\line(-1,0){3.5}}
\put(4,3.81){\line(3,1){2}}
\put(4,3.81){\line(-1,1){2}}
\put(3.2,3.3){\LARGE $v_1$}
\put(2,4){\LARGE $e_1$}
\put(2.9,5){\LARGE $e_2$}
\put(4.6,4.4){\LARGE $e_3$}
\put(6,4.42){\circle*{.5}}
\put(6,4.42){\line(3,-2){2}}
\put(6,4.41){\line(3,-2){2}}
\put(5.3,3.6){\LARGE $v_2$}
\put(7,3.8){\LARGE $e_4$}
\put(6,4.42){\line(4,1){3}}
\put(6,4.41){\line(4,1){3}}
\put(8.2,5.4){\LARGE $e_5$}
\put(6,4.42){\line(-1,1){1.6}}
\put(6,4.41){\line(-1,1){1.6}}
\put(5.25,5.25){\LARGE $e_6$}
\put(.5,3.8){\circle*{.2}}
\put(.7,3.3){\LARGE $v_3$}
\put(2,5.8){\circle*{.2}}
\put(2.2,5.8){\LARGE $v_4$}
\put(4.4,6){\circle*{.2}}
\put(4.55,6){\LARGE $v_5$}
\put(9,5.15){\circle*{.2}}
\put(9.2,5.15){\LARGE $v_6$}
\put(8,3.1){\circle*{.2}}
\put(8.2,3.1){\LARGE $v_2$}
\put(4,.8){\line(-1,0){3.5}}
\put(4,6.8){\line(-1,0){3.5}}
\put(.5,.8){\line(0,1){6}}
\put(.8,6){\LARGE $\Omega_1$}
\put(4,6.8){\line(-1,1){2}}
\put(2,6.8){\line(0,1){2}}
\multiput(2,6.6)(0,-1){4} {\line(0,-1){0.2}}
\multiput(3.5,1.3)(-1,1){2} {\line(-1,1){0.4}}
\put(2.1,7.6){\LARGE $\Omega_2$}
\put(4,.8){\line(3,1){2}}
\put(4,6.8){\line(3,1){2}}
\put(5.2,6.6){\LARGE $\Omega_3$}
\put(8,.1){\line(0,1){6}}
\put(6,7.45){\line(3,-2){2}}
\put(6,1.45){\line(3,-2){2}}
\put(7.2,5.8){\LARGE $\Omega_4$}
\put(6,7.48){\line(4,1){3}}
\put(6.4,1.55){\line(4,1){0.4}}
\put(7.2,1.75){\line(4,1){0.4}}
\put(8,1.95){\line(4,1){1}}
\put(9,2.2){\line(0,1){6}}
\put(8.2,7.4){\LARGE $\Omega_5$}
\put(6,7.48){\line(-1,1){1.6}}
\put(5.6,1.85){\line(-1,1){0.4}}
\put(4.8,2.65){\line(-1,1){0.4}}
\put(4.4,6.95){\line(0,1){2.12}}
\multiput(4.4,6.6)(0,-1){4} {\line(0,-1){0.2}}
\put(4.5,7.8){\LARGE $\Omega_6$}
\linethickness{0.8mm}
\put(8.3,.9){$V=\{v_i\}$}
\put(9.92,.9){, $E=\{e_i\}$}
\multiput(9.25,0.5)(0.4,0){2} {\line(1,0){0.2}}
\put(8.3,0.4){$\Sigma=$}
\put(9.92,0.4){, $R_j=\Omega_j$}
\end{picture}

\begin{example}
The cube in the example \ref{cube} is not locally homeomorphic to an elementary ramified space at the corner points. It is a LEP space if all the $2^n$ corner points are contained in $N_E^\star=\pd \O$.
\end{example}
\section{The differential structure of a ramified manifold}\label{section2}
In this section we extend the notion of tangent space to a differentiable ramified manifold. In fact, the  interpretation of
tangent vectors as equivalence classes of curves in $M$ can be easily transferred to ramification points.

Throughout this section, $M$ and $\Sigma$ stand respectively for a $n$-dimensional differentiable ramified manifold and for its ramification set.
Let us now introduce some definitions regarding the differential structure of $M$.

\begin{definition}\label{ldiff}
A continuous function $f:M\to\R$ is said to be $C^l$-differentiable
at $x\in M$ if for any local chart $(U,X)$ around $x$, the function $f\circ X^{-1}$ is $C^l$-differentiable in sense of definition \ref{diffeo}.
\end{definition}
\begin{definition}
Let $x\in \Sigma$ and $r=r(x)$.
Let $\g:(-\e,\e)\to M$ with $\g(0)=x$ be a continuous curve and $j\in \{1,\dots,r\}$. We say that $\g$ reaches $x$ from the branch $j$ whenever
there exists a chart $(U,X)$ with $x\in U$ and $\d>0$ such that
\begin{equation}\label{curve}
    \tilde \g(t)=(X\circ \g)(t)\in \cR^n_{r,j}\quad\text{for all $t\in (-\d,0)$}.
\end{equation}
\end{definition}
We denote by $\cC_j(x)$ the set of all the curves reaching $x$ from the branch $j$ and we set
\[\cC(x)=\cup_{j\in J}\cC_j(x).\]
\begin{definition}\label{tangentspace}
Let $x\in \Sigma$, $r=r(x)$ and  $\g_1$, $\g_2\in \cC(x)$. We say that $\g_1$ and $\g_2$ are equivalent  if for all functions $f:M\to\R$
which are $C^\infty$-differentiable at $x$ we have
\[(f\circ \g_1)'_-(0)=(f\circ \g_2)'_-(0)\]
where the  derivatives are left-sided. We denote the set of equivalence classes by $T_xM$, the tangent space of $M$ at $x$, and
we say that  $\xi\in T_xM$ is a $j$-tangent vector at $x$, $1\le j\le r$, if $\xi$ contains a curve reaching $x$ from
the branch $j$. We set
\[\xi(f):=(f\circ \g)'_-(0)\qquad \g\in\xi.\]
The set of all $j$-tangent vectors at $x$ ($j$-tangent space at $x$) is denoted by $T_x^j M$.  The set $T_x\Sigma:=\cap_j T_x^j M$ is called the $\Sigma$-tangent space at $x$ and any $\xi\in T_x\Sigma$ is said a  $\Sigma$-tangent vector at $x$.
\end{definition}
\begin{remark}
If $x\in\Sigma$, the tangent space $T_xM$ is not a vector space. Instead it can be identified with an elementary ramified space $\cR^n_r$, where
\[\cR^n_{r,j}=T_x^j M, \, 1\le j\le r,\qquad \Sigma^n_r=T_x\Sigma.\]
Hence $T_x\Sigma\subset T^j_x M$ can be identified with $\R^{n-1}$ and $T_x^j M$ with $\R^n_{\ge 0}$.
\end{remark}
\begin{definition}
Let $1\le j\le r$, $f:M\to\R$ continuously differentiable at $x$ (see def. \ref{ldiff}) and $\xi_1,\dots,\xi_n$ a basis of $T_x^j M$. We define the $j$-gradient
$D^j f\in T_x^j M$ of  $f$ at $x$  by
\begin{equation}\label{gradient}
D^jf(x):=\sum_{i=1}^n\xi_i(f)\xi_i.
\end{equation}
\end{definition}
We  consider the case of an elementary ramified space and we introduce some notations for the derivatives at the ramification set.
\begin{definition}\label{derivativeLP}
Let $r\ge 3$ and let $x\in\Sigma_r^n\subset\cR^n_r$. Let $u:\cR^n_r\to\R$ be continuously differentiable at $x$. We denote  by $\pd_1 u(x),\dots,
\pd_{n-1}u(x)$ the directional derivatives of $u$ at $x$ with respect to the canonical basis $e_1,\dots, e_{n-1}$ of $\Sigma^n_r\equiv \R^{n-1}$.
For $1\le j\le r$ we denote by $\pd_{\nu_j}u(x)$ the directional derivative of $u$ at $x$ with respect to the inward unit normal $\nu_j$ of
$\cR^n_{r,j}\equiv \R^n_{\ge 0}$ at $x$. See figure 2.
\end{definition}

We now restrict our attention to LEP spaces, which have the important property that around any given point we can always choose a chart induced by the canonical identification with the Euclidean space $\R^n$ or a suitable elementary ramified space. This is stated in the following proposition.
\begin{proposition}\label{canonical}
Let $\O$ be a LEP space and $\Sigma$ its ramification set. For any $x\in\O$, there is a neighborhood $V_x$ of $x$ and a canonical identification
$\im_x:V_x\to \im_x(V_x)$
where $\im_x(x)=0$ and
\begin{itemize}
  \item[i)] if $x\not\in\Sigma$ then  $\im_x(V_x) \subset \R^n$,
  \item[ii)] if $x\in\Sigma$ then $\im_x(V_x) \subset \cR^n_{r(x)}$ with
$\im_x(V_x\cap \Sigma)\subset\Sigma^n_{r(x)}$.
\end{itemize}
In the latter case, $\im_x$ induces a bijective map $\cI_x$ between the index set  $\Inc_x$ and the set $\{1,\dots,r(x)\}$.
\end{proposition}

We now consider derivatives  of a function on a LEP space at ramification points. For any $x\in \O$, we always fix a canonical identification chart $(V_x,\im_x)$ as defined
in proposition \ref{canonical} and all the concepts will be expressed in terms of the chart  $\im_x$ for sake of simplicity. However it is easy to verify that they in fact do not depend on the choice of specific chart.
\begin{definition}\label{derivativeLPbis} Let $\O$ be a LEP space with rampification set $\Sigma$.
Let $x\in\Sigma$, $r=r(x)$. Let  $V\subset\O$ be a neighborhood of $x$ and let $u:V\to\R$ be a function which is continuously differentiable at $x$.
Following the notation of definition \ref{derivativeLP}, we set
\begin{align}
 \pd_i u(x)&:=\pd_i(u\circ\im_x^{-1})(0),& i=1, \dots,n-1\nonumber\\
 \pd_{\nu_j}u(x)&:=\pd_{\nu_{\cI_x(j)}}(u\circ\im_x^{-1})(0),&j\in\Inc_x\nonumber\\
 \pd^j u(x)&:=(\pd_1 u(x),\dots,\pd_{n-1} u(x),\pd_{\nu_j}u(x))\label{cangradient}
\end{align}
where $\cI_x$ is defined as in proposition \ref{canonical}
(note that for each $j\in \Inc_x$ the collection $\{\pd_1,\dots\pd_{n-1},\pd_{\nu_j}\}$ form a basis of $T^j_x\O$).
 \end{definition}

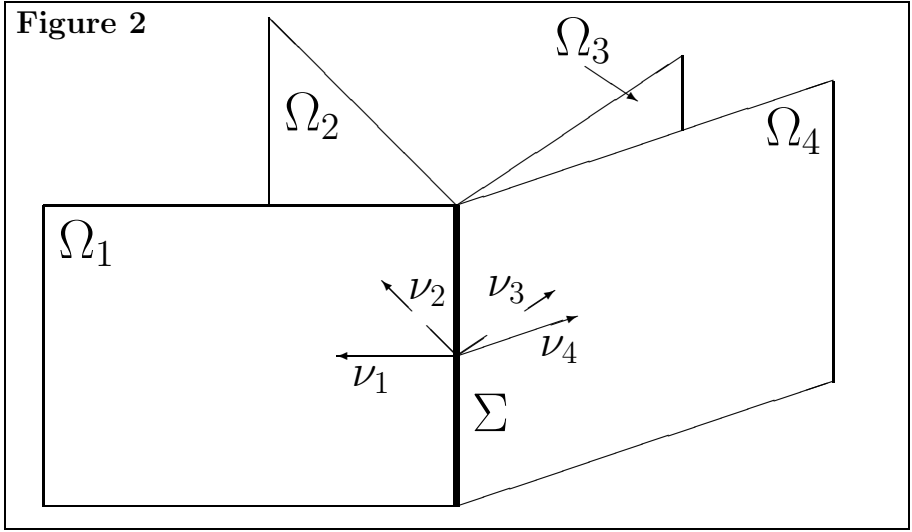
 \setlength{\unitlength}{1cm}
\begin{picture}(14,7)
 \put(0,0){\framebox(12,7)}
\put(0,6.6){ {\bf Figure 2}}
\linethickness{0.8mm}
\put(6,.3){\line(0,1){4}}
\put(6.2,1.3){\LARGE $\Sigma$}
\linethickness{0.2mm}
\put(6,.3){\line(-1,0){5.5}}
\put(6,4.3){\line(-1,0){5.5}}
\put(.5,.3){\line(0,1){4}}
\put(.7,3.6){\LARGE $\Omega_1$}
\put(6,2.3){\vector(-1,0){1.6}}
\put(4.6,1.9){\Large{$\nu_1$}}
\put(6,4.3){\line(-1,1){2.5}}
\put(3.5,4.3){\line(0,1){2.5}}
\put(3.7,5.3){\LARGE $\Omega_2$}
\put(6,2.3){\line(-1,1){0.4}}
\put(5.4,2.9){\vector(-1,1){0.4}}
\put(5.35,3.05){\Large{$\nu_2$}}
\put(6,4.3){\line(3,2){3}}
\put(9,5.3){\line(0,1){1}}
\put(7.3,6.3){\LARGE $\Omega_3$}
\put(7.7,6.15){\vector(3,-2){.7}}
\put(6,2.3){\line(3,2){0.4}}
\put(6.9,2.9){\vector(3,2){0.4}}
\put(6.4,3.1){\Large{$\nu_3$}}
\put(6,.3){\line(3,1){5}}
\put(6,4.3){\line(3,1){5}}
\put(11,1.95){\line(0,1){4}}
\put(10.1,5.1){\LARGE $\Omega_4$}
\put(6,2.3){\vector(3,1){1.6}}
\put(7.1,2.3){\Large{$\nu_4$}}
\end{picture}

For any function $u:\O\to\R$ and each $j\in J$ we denote by $u^j:\tilde \O_j\rightarrow \R$ the restriction of $u$ to $\tilde \O_j$, i.e.
\[u^j(x):=u\circ \im_x^{-1}(0)\qquad \text{for $x\in\tilde\O_j$}\] We denote by $C(\O)$ the space of continuous function on $\O$. This in particular implies that $u^j\in C(\tilde \O_j)$ and
\[u^j(x)=u^k(x) \qquad\text{for any $x\in \Sigma$, $j,k\in Inc_x$.}\]
In a similar way we define the space of upper semi-continuous functions  $\text{USC}(\O)$ and  the space of
lower semi-continuous functions $\text{LSC}(\O)$, respectively. \par
In \cite{sc}, we introduced the concept of $(j,k)$-test function on a topological network $\G$, treating two edges $j$ and $k$  incident  at a  vertex $v_i$  as one connected
edge and imposing  that the derivatives in the direction of the incident edges,   taking into account their orientations,  coincide at $v_i$.
In other terms, a test function, considered as a function defined on $e_j\cup e_k$,  is differentiable at the interior point $v_i$.\par
Here we follow a similar idea for LEP spaces, linking the two normal derivatives of a test function  for a given couple of branch manifolds
incident at  a point $x\in\Sigma$.
\begin{definition}\label{testfunction}
Let $\phi\in C(\O)$, $x\in \Sigma$, $k,l\in\Inc_x$, $k\neq l$. Then $\phi$ is said to be $(k,l)$-differentiable at $x$ if $\phi$ is $C^1$-differentiable at $x$ and if we have
\begin{equation}\label{differentiability}
    \pd_{\nu_k}\phi(x)+\pd_{\nu_l}\phi(x)=0.
\end{equation}
\end{definition}
\begin{definition}
Let $u:\O\to \R$ and let $\phi$ be $C(\O)$.
\begin{itemize}
  \item Let $j\in J$ and let $x\in\O_j$. We say that $\phi$ is an upper (lower) test function of $u$ at $x$ if $\phi$ is $C^1$-differentiable at $x$ and $u-\phi$ attains a local maximum (minimum) at $x$.
 \item Let $x\in \Sigma$ and $k,l\in\Inc_x$, $k\neq l$.  We say that $\phi$ is a $(k,l)$ upper (lower) test function of $u$ at $x$ if $\phi$ is $(k,l)$-differentiable at $x$ and $u-\phi$ attains a local maximum (minimum) at $x$ with  respect to $\O_{k,l}:=\bar\O_k\cup\bar\O_l$.
\end{itemize}
\end{definition}
\section{Viscosity solutions}\label{section3}
Since now on, $\O$ and $\Sigma$ stand respectively for a LEP space and for its ramification set.
We introduce the class of Hamilton-Jacobi equations of eikonal type we consider in this paper.
An Hamiltonian $H=(H^j)_{j\in J}$ is a family of mappings
$H^j:\tilde\O_j\times T_x^j\O\to \R$ with $H^j\in C^0(\tilde \O_j\times \R^n)$ (recall: $\pd_R\O_j:=\pd\O_j\cap \Sigma$, $\tilde \O_j=\O_j\cup\pd_R\O_J$).
By means of the canonical identification map $ \im_x$ around a fixed point  $x\in\tilde\O_j$ we can think of $H^j$ as a mapping
$\td H^j:V\times\R^n\to \R$ defined by the identification
 \begin{equation}\label{Hcanonical}
 H^j(y ,D^j u(y))=\tilde H^j(\im_x(y),\pd^j u(y))\qquad \forall y\in V_x
 \end{equation}
 (see \eqref{gradient} and \eqref{cangradient})
 where $V$ is a neighborhood of $0\in\R^n$ or $0\in\R^n_{\ge 0}$ provided that $x\in\O_j$ or $x\in\Sigma$, respectively. In the sequel
we will speak of $H^j$ under the canonical  identification (around $x$) whenever we refer to $H^j$ in the sense of \eqref{Hcanonical}. We assume
 the Hamiltonian $H=(H^j)_{j\in J}$ fulfills the following properties
\begin{eqnarray}
   &&H^j \text{ is continuous on } \tilde\O_j\times \R;\label{H15}\\
   && H^j(x,p)\to +\infty \quad \text{as $|p|\to \infty$ for $x\in\tilde\O_j$;}\label{H3}\\
   &&\text{for each }x\in \Sigma,\, j\in \Inc_x,\,  p\mapsto H^j\left(x,(p_1,\dots, p_{n-1},\cdot\right) \text{ is not decreasing for $p\geq 0$};\label{H2}\\
  && \text{for each }x\in \Sigma,\, j,k\in \Inc_x,\, H^j(x,p)=H^k(x,p)\quad\forall p\in \R^n;\label{H4}\\
&& \text{for each }x\in \Sigma,\, j\in \Inc_x,\, H^j(x,(p',p_n))=H^j(x,(p',-p_n))\quad\forall p'\in \R^{n-1}, p_n\in \R.\label{H5}
\end{eqnarray}
\begin{remark}\label{comm_on_hyp}
Assumptions \eqref{H15}--\eqref{H3} are standard conditions in viscosity solution theory (see f.e.  \cite{i})
to ensure existence and uniqueness of the solution. Assumptions \eqref{H4} and \eqref{H5} represent compatibility conditions across
the ramification set; the former guarantees a continuity condition at $x\in\Sigma$ for the Hamiltonians
defined on two different branches while the latter states the invariance with respect to orientation of the inward normal $\nu_j$.
Under hypotheses~\eqref{H3} and~\eqref{H5}, assumption~\eqref{H2} is fulfilled provided that, for $x\in\Sigma$, $H^j(x,\cdot)$ is convex.
Observe that thanks to the identification \eqref{Hcanonical}, \eqref{H2}--\eqref{H5} induce corresponding properties
for the Hamiltonian $\td H=\{\td H^j\}_{j\in J}$.
\end{remark}
\begin{example}\label{eikonal}
A typical example of Hamiltonian satisfying the previous assumptions is given by the family
$H^j(x,p)=|p|^2-f^j(x)$ where the functions $f^j:\bar\O_j\to \R$ are continuous, non negative and satisfies
the compatibility condition $f^j(x)=f^k(x)$ if $x\in\pd_R\O_j\cap\pd_R\O_k$.
\end{example}
We  introduce the definition of viscosity solution for the Hamilton-Jacobi equation  of eikonal type
\begin{equation}\label{HJ}
    H(x,Du)=0,\qquad x\in \O.
\end{equation}
For $x\in\Sigma$, we define by  $\pi_j:T_x^j\O\to T_x\Sigma$ the projection on the tangent space of $\Sigma$, i.e.
\[
\pi_j(p):=\sum_{m=1}^{n-1}p_m\pd_m\qquad\text{for}\quad p=p_n\pd_{\nu_j}+\sum_{m=1}^{n-1}p_m \pd_m
\]
\begin{definition}\label{1:def3}\hfill\\
A function $u\in\text{USC}(\O)$ is called a (viscosity) subsolution of \eqref{HJ} in $\O$ if the following holds:
\begin{itemize}
  \item[i)] For any  $x\in \O_j$, $j\in J$, and for any  upper test function $\phi$ of $u$ at $x$  we have
 \[H^j(x, D^j \phi(x))\le 0.\]
  \item [ii)] For any $x\in \Sigma$, for any $j,k\in\Inc_x$ and for any $(j,k)$ upper test function $\phi$ of $u$ at $x$ we have
 \[H^j(x, D^j \phi(x))\le 0.\]
\end{itemize}
 A function $u\in\text{LSC}(\O)$ is called a (viscosity) supersolution of \eqref{HJ} in $\O$ if the following holds:
\begin{itemize}
  \item[i)] For any  $x\in \O_j$, $j\in J$,  and  for any  lower test function $\phi$ of $u$ at $x$  we have
 \[H^j(x, D^j \phi(x))\ge 0.\]
  \item [ii)] For any $x\in\Sigma$, $p'\in T_x\Sigma$ (see definition~\ref{tangentspace}) and  $j\in \Inc_x$, there exists $k\in \Inc_x$, $k\neq j$ such that
 \[H^j(x, D^j \phi(x))\ge 0\]
 for any $(j,k)$ lower test function $\phi$ of $u$ at $x$ satisfying   $\pi_j(D^j\phi(x))=p'$.
\end{itemize}
A continuous function $u\in C(\O)$ is called a (viscosity) solution of  \eqref{HJ} if it is both a viscosity subsolution and a viscosity supersolution.
\end{definition}
\begin{remark}\label{1:r2}
{\it i)}
Near $x\in\Sigma$ the set $\O_k\cup \O_j$ is locally diffeomorphic  to  $\R^n$  with a tangent space
given by $T_x^j\O\cup T_x^k\O$. The tangent space is composed by vectors $(p',p_n)$ with
$p'\in T_x\Sigma\subset T_x^j\O\cap T_x^k\O$ and $p_n\perp T_x\Sigma$. Taking into account the condition of differentiability \eqref{differentiability}, we see that if
$\phi$ is a  $(j,k)$-test function of $u$ at $x\in\Sigma$, then   by \eqref{H4}-\eqref{H5},
\[
H^j(x, D^j \phi(x))=H^k(x,D^k\phi(x)).
\]
In other terms, at a ramification point, the equation \eqref{HJ} can be equivalently replaced by one of the two couples $(H^j(x,p), T_x^j\Omega)$ and $(H^k(x,p),T_x^k\Omega)$.\\
{\it ii)}
The definitions of subsolution and supersolution are not symmetric at a ramification point. It turns out that solutions of eikonal equations are
distance-like functions from the boundary (see section \ref{section6}); this definition of supersolution reflects the idea these functions
 have to be solutions  of \eqref{HJ}. Since  there is always a shortest path from a ramification point to the boundary, for any branch
$\O_j$ and for any $x\in\Sigma\cap\O_j$ we can connect the $x$ to the boundary contained in the branch $\O_k$ by a shortest path. Then $(j,k)$-lower test functions satisfies the definition of supersolution.\\
It is also  worth to observe that, for $r(x)=2$, our definition of solution (resp., sub- or super-solution) coincides with the standard notion of viscosity solution (resp., sub- or super-solution).
\end{remark}
\begin{remark}\label{1:r2bis}
The  definition of viscosity solution is not affected by the fact that $x$ belongs or not to the ramification set $\Sigma$.
At ramification points, it requires a
suitable class of test functions (different from the classical one) with appropriate differentiability properties.
In this respect, our approach is different from the one in \cite{acct,bh,imz} which require to suitably redefine  the
Hamiltonian at ramification points.
\end{remark}
%
%

\section{A comparison result}\label{section4}
This section is devoted to the proof of a comparison theorem for \eqref{HJ}.
The coerciveness  of the Hamiltonian (stated in~\eqref{H3}) implies the  Lipschitz-continuity of   subsolutions; such a regularity which will be exploited in the comparison theorem.
\begin{lemma}\label{lip}
Let $K$ be a compact subset of $\O$ and let $u$ be a subsolution of \eqref{HJ}. Then there exists a constant $C_K$ depending only on $K$ such that
\[
|u(x)-u(y)|\le C_Kd(x,y).
\]
\end{lemma}
The proof is standard in viscosity theory and we refer the reader to~\cite[Prop.II.4.1]{bcd}.

%
\begin{theorem}\label{Comparison}
  Let $v\in LSC(\bar \O)$ and $u\in USC(\bar \O) $ be respectively a supersolution to  \eqref{HJ} and a subsolution of
\begin{equation}\label{HJsbis}
    H(x,D u)=f(x),\qquad x\in\O,
\end{equation}
with $f\in C(\O)$ and  $f(x)<0$ for all $x\in\O$. If $u\le v$ on $\partial\O$, then $u\le v$ in $\bar \O$.
\end{theorem}
\begin{proof}
Assume by contradiction that there exists $z\in\O$ such that
\begin{equation}\label{5:1}
    u(z)-v(z)=\max_{\bar\O}\{u-v\}=\delta>0.
\end{equation}
For $\e>0$ define $\Phi_\e:\bar\O\times\bar\O\to \R$ by
\[\Phi_\e(x,y):=u(x)-v(y)- \e^{-1}d(x,y)^2\]
where $d(x,y)$ is the geodesic distance between two points $x$ and $y$ on the space $\bar \O$.
Since $\Phi_\e$ is an  upper semi-continuous function, there exists a maximum point $(x_\e,y_\e)$ for $\Phi_\e$ in $\bar\O^2$. By
$\Phi_\e(z,z)\le \Phi_\e(x_\e,y_\e)$ we get
 \begin{equation}\label{5:2}
 \e^{-1}d(x_\e,y_\e)^2\le u(x_\e)-v(y_\e)-\delta\le M,
 \end{equation}
for some $M\in\R$. Hence
\begin{equation}\label{5:3}
    \lim_{\e\to 0} d(x_\e,y_\e)=0.
\end{equation}
It follows that there exists $x_0\in\bar \O$ such that $x_\e,y_\e \to x_0$.
Owing to \eqref{5:2}, there holds: $\d\leq u(x_\e)-v(y_\e)$; passing to the limit, we infer $0<\d\leq u(x_0)-v(x_0)$ and, in particular, $x_0\notin \partial \O$. Whence, for $\e$ sufficiently small, we get: $x_\e, y_\e,x_0\in\O$.
 By \eqref{5:2} and the Lipschitz
continuity of $u$ (see lemma \ref{lip}) we get
\begin{equation*}
 \e^{-1}d(x_\e,y_\e)^2\le u(x_\e)-u(y_\e)+u(y_\e)-v(y_\e)-\delta\le Ld(x_\e, y_\e)
 \end{equation*}
and  therefore
\begin{equation}\label{5:4}
\lim_{\e \to 0^+}\e^{-1}d(x_\e,y_\e)=0.
\end{equation}
Since $x_0\in \O$, there exists $r>0$ such that
\begin{equation}\label{5:4bis}
B_r(x_0)\cap \pd\O=\emptyset
\end{equation}
and for $\e$ sufficiently small,
$x_\e, y_\e \in B_r(x_0)$.
Since $\O$ is compact and composed by a finite number of branch $\O_j$, we have
\[\inf_{j\in J}\inf_{k:\bar\O_j\cap\bar\O_k=\emptyset}d(\O_j,\O_k)>0.\]
Hence, for $\e$ sufficiently small, we can assume that if $x_\e\in \O_{j_{\e}}$ and $y_\e\in \O_{k_{\e}}$, then
$\bar\O_{j_{\e}}\cap\bar\O_{k_{\e}}\neq \emptyset$.

>From now on we set $x:=x_\e$, $y:=y_\e$ and we define
 $\phi_x(\cdot):=\e^{-1}d(x,\cdot)^2$ and $\phi_y(\cdot):=\e^{-1}d(\cdot,y)^2$ and we work with the canonical identification
 \eqref{Hcanonical}.
  \par
\emph{Case 1. $x, y\in \td\O_j$ for some $j\in J$:} Consider
 $\im_{x_0}:V_{x_0}\to \im_{x_0}(V_{x_0})$ and    set
\[\tilde x=\im_{x_0}(x),\quad \tilde y=\im_{x_0}(y).\]
If $x \in\O_j$, then $\phi_x(\cdot)$ is differentiable at $ y$, and
\begin{align*}
    \pd^j \phi_x(\tilde y)=\e^{-1}d(x,y)\frac{\tilde y-\tilde x}{|\tilde y-\tilde x|}.
\end{align*}
If  $x\in\Sigma$, then  we have
\begin{align*}
 &\pd^j \phi_x(\tilde y)=\e^{-1}d(x,y)\frac{\tilde y-\tilde x}{|\tilde y-\tilde x|}\\
 &(\pd_1 \phi_x(\tilde y),\dots,\pd_{n-1}\phi_x(\tilde y))\in T_x\Sigma_{r(x)}^n\\
 &\pd_{\nu_k} \phi_x(\tilde y)=-\pd_{\nu_j} \phi_x(\tilde y)\quad\text{for any  $k\in \Inc_x$, $k\neq j$,}
\end{align*}
We conclude that for any  $k\in \Inc_x$, $k\neq j$,  $\phi_x(\cdot)$ is $(j,k)$-differentiable at $y$.
A corresponding property holds for $\phi_y(\cdot)$.\par
Since   $u-\phi_y$ has a maximum point at $x$ and $v+\phi_x$ has a minimum point at $y$, we get
\begin{align*}
    &H^j( x,D^j\phi_y(x))\le  f( x) \\
    &H^j(y,-D^j\phi_x(y))\ge 0.
\end{align*}
Moreover
\[H^j(y,-D^j\phi_x(y))=\td H^j(\tilde y,\e^{-1}d(x,y)\frac{\tilde y-\tilde x}{|\tilde y-\tilde x|} )=H^j(y,D^j\phi_y(x)) .\]
Fix $R>0$ and denote by $\omega_j$,   the modulus of continuity of $H^j$ with respect to $(x,p)\in\bar \O_j\times [-R,R]^n$.
Hence, for sufficiently small $\e>0$ we have
\begin{align*}
  \eta\le -f( x)\le   H^j\left (y,D^j\phi_y(x) \right) -H^j\left(  x, D^j\phi_y(x)\right)
\le \omega_j(d(x,y))
\end{align*}
for some $\eta>0$. By \eqref{5:4} we get a contradiction for $\e\to 0$.

\emph{Case 2.  $x\in   \O_j$, $y\in  \O_k$  for some $j,k \in J$,  $j\neq k$:}
Consider
 $\im_{x_0}:V_{x_0}\to \im_{x_0}(V_{x_0})$,    set $r=r(x_0)$ (see def. \ref{derivativeLPbis}) and
\[\tilde x=\im_{x_0}(x)\in  \cR^n_{r,\cI_{x_0}(j)},\quad \tilde y=\im_{x_0}(y)\in \cR^n_{r,\cI_{x_0}(k)}.\]
Since $x\in   \O_j$, $y\in  \O_k$, by   \eqref{5:4bis} there exists $z\in \td\O_j\cap\td\O_k\subset \Sigma$ such that
$d(x,y)=d(x,z)+d(z,y)$ and, defined $\td z=\im_{x_0}(z)\in\Sigma^n_r$, we have
    \begin{align*}
    &\pd^k \phi_x(\tilde y)=\e^{-1}d(x,y)\frac{\tilde y-\tilde z}{|\tilde y-\tilde z|}\\
    &\pd^j \phi_y(\tilde x)=\e^{-1}d(x,y)\frac{\tilde x-\tilde z}{|\tilde x-\tilde z|}.
    \end{align*}
Moreover
\begin{align}
     &\pd_i \phi_x(\tilde y)=-\pd_i \phi_y(\tilde x) \qquad i=1,\dots,n-1\label{5:5}\\
    &\pd_{\nu_k} \phi_x(\tilde y)=\pd_{\nu_j} \phi_y(\tilde x)\label{5:6}
\end{align}
taking into account the definition of $\nu_j$, $\nu_k$.
As in Case 1,  $\phi_x$ and $\phi_y$ are $(j,k)$-differentiable  and
$u-\phi_y$ has a maximum point at $x$ and $v+\phi_x$ has a minimum point at $y$. Therefore
\begin{align*}
    &H^j( x,D^j\phi_y(x))\le f( x) \\
   &H^k(y,-D^k\phi_x(y))\ge 0.
\end{align*}
Fix $R>0$ and denote by $\omega$,   the modulus of continuity of every $H^l$ with respect to $(x,p)\in\bar \O_l\times [-R,R]^n$.
For sufficiently small $\e>0$ we have
\begin{align*}
  \eta&\le -f( x)\le H^k(y,-D^k\phi_x(y))- H^j( x,D^j\phi_y(x))\le H^k(y,-D^k\phi_x(y))-H^k(z,-D^k\phi_x(y))\\
  &H^j(z,D^j\phi_y(x))-H^j(x,D^j\phi_y(x))+ H^k(z,-D^k\phi_x(y))- H^j( z,D^j\phi_y(x))\\
  &\le\omega(d(y,z))+\omega(d(x,z))-H^j(z,D^j\phi_y(x))+ H^k(z,-D^k\phi_x(y))
\end{align*}
By \eqref{H4}, \eqref{H5},\eqref{5:5} and  \eqref{5:6} we get
\begin{align*}
H^k(z,-D^k\phi_x(y))&= \td H^k\big(\td z, (-\pd_1  \phi_x(\tilde y),\dots,-\pd_{n-1} \phi_x(\tilde y), -\pd_{\nu_k}\phi_x(\tilde y)\big)\\
&= \td H^j\big(\td z, (\pd_1  \phi_y(\tilde x),\dots,\pd_{n-1} \phi_y(\tilde x), -\pd_{\nu_j}\phi_y(\tilde x)\big)\\
&= \td H^j\big(\td z, (\pd_1  \phi_y(\tilde x),\dots,\pd_{n-1} \phi_y(\tilde x), \pd_{\nu_j}\phi_y(\tilde x)\big)=H^j( z,D^j\phi_y(x))
\end{align*}
Hence
\[\eta \le -f( x)\le \omega(d(y,z))+\omega(d(x,z))\]
which gives a contradiction for $\e\to 0^+$ since by \eqref{5:3}, $  x,  y\to  x_0$ for $\e\to 0$.
\end{proof}
\begin{remark}
For $H=(h^j(p)-f^j(x))_{j\in J}$, with $h^j$ positively homogeneous, the comparison principle can be also obtained by means of Kru\v zkov transform (see \cite{bcd}).
\end{remark}
\section{Existence}\label{section5}
In this section we prove existence of a solution via   a Perron's method.
We define
$u^*,u_*:\O\to [-\infty,\infty]$ by
\[u^*(x)={\lim\sup}_{r\to 0}\{u(y):\,d(x,y)\le r\}\quad\text{and}\quad u_*(x)={\lim\inf}_{r\to 0}\{u(y):\,d(x,y)\le r\}. \]
\begin{lemma}\label{Perronsub}
Let $\cV$ be an arbitrary set of viscosity subsolutions of \eqref{HJ}. Define  $u(x):=\sup_{v\in \cV} v(x)$ for all $x\in\O$ and assume that
$u^*(x)<\infty$ in $\Omega$. Then $u^*$ is a viscosity subsolution  of \eqref{HJ}.
\end{lemma}
\begin{proof}
We assume that $x\in \Sigma$, otherwise we can use standard arguments in viscosity solution theory. We consider a $(j,k)$-upper test function $\phi$ at $x$ and we have to show that
\begin{equation}\label{E:1}
    H^j(x,D^j\phi(x))\le 0.
\end{equation}
Let $\eta>0$ be such that $u^*-\phi$ has a local maximum point at $x$ in $\bar B=B_\eta(x)\cap(\bar\O_j\cup\bar\O_k)$ where
$B_\eta(x)=\{y\in\O:\, d(x,y)\le \eta\}$.
Since $\O$ is a LEP space we can assume that $\eta>0$ is sufficiently small such that $B_\eta(x)\cap\Sigma$ is contained in a $(n-1)$-dimensional affine linear subspace of $\R^{n+1}$. We consider the function $\phi_\d(y)=\phi(y)+\delta d_x(y)^2$, for   $\d>0$, where $d_x(\cdot)=d(x,\cdot)$. Then $u^*-\phi_\d$ has a strict maximum point at $x$.\\
Let $\eta_n\to 0$
 for $n\to \infty$ and $x_n\in B_n:=B_{\eta_n}(x)$ such that $\sup_{B_n} u -u(x_n)\le 1/n$. By the definition of $u$, there is $u_n\in \cV$ such that $u(x_n)-u_n(x_n)\le 1/n$. Hence
\begin{equation}\label{E:2}
\sup_{B_n}u\ge u_n(x_n)-\frac{1}{ n}>\sup_{B_n}u-\frac{2}{ n}.
\end{equation}
and
\begin{equation}\label{E:3}
   u^*(x)=\lim_n\, \sup_{B_n} u=\lim_n u_n(x_n).
\end{equation}
Let $y_n$ be  a maximum point for $u_n-\phi_\d$ in $\bar B$. Up to a subsequence, we can assume that $y_n\to z\in \bar B$.
By  lemma \ref{lip} there exists a constant $C=C_{\bar B}$ such that
\begin{align*}
     u(y_n)-\phi_\d(y_n)&\ge u_n(y_n)-\phi_\d(y_n)\ge u_n(x)-\phi_\d(x)\ge u_n(x_n)-\phi_\d(x_n)-C|x-x_n|
\end{align*}
Taking the $\limsup$  for $n\to \infty$ in the previous inequality, by \eqref{E:3} and the definition of $\lim\sup^*$ we get
\begin{equation}\label{E:4}
u^*(z)-\phi_\d(z)\ge     u^*(x)-\phi_\d(x)
\end{equation}
and therefore $x=z$ since $u^*-\phi_\d$ has a strict maximum point at $x$. Moreover
\begin{equation}\label{E:4bis}
\lim y_n=x\quad \text{and}\quad   \lim_n u_n(y_n)=u^*(x).
\end{equation}
We distinguish two cases.\\
\emph{Case 1.}
There are infinitely many $y_n$ with $y_n\not\in\Sigma$. Then there exists a subsequence -also denoted with $\{y_n\}_n$- which is completely contained in either $\O_k$ or $\O_j$. Without restriction we assume that $y_n\in\O_k$ for all $n$. Since $\phi_\d$ is an upper test function
for  $u_n$ at $y_n$ and $\phi_\d$ is differentiable at $y_n$, we get
\begin{equation}\label{E:5}
    H^k(y_n, D^k\phi_\d(y_n))=\td H^k\left(\td y_n, \pd^k\phi(\td y_n)+ 2d(x,y_n)\frac{\td y_n}{|\td y_n|}\right)\le 0
\end{equation}
where $\td y_n=\im_x(y_n)$ (recall  that $\im_x$ is the canonical identification around $x$ with $\im_x(x)=0$).\\
\emph{Case 2.} By possible truncating the sequence $\{y_n\}_n$ we can assume that $y_n\in \Sigma\cap \bar B$ for any $n$. Since $\Sigma$
is contained in a $(n-1)$-dimensional  linear subspace of $\R^{n+1}$, we have
\[\pd_{\nu_j}d_x(y)=0 \quad \forall y\in \Sigma\cap \bar B,\,j\in \Inc_x.\]
In particular $d_x$ is $(j,k)$-differentiable in $\Sigma\cap \bar B$ and therefore
$\phi_\d$  is  $(j,k)$-differentiable  in $\Sigma\cap B$
with $\pd_m\phi_\d(0)=\pd_m\phi(0)$ for either $m=j$ or $m=k$. Since $\phi_\d$ is a $(j,k)$-upper test function
for $u_n$ at $x$ we get
\begin{equation}\label{E:6}
   H^m(y_n, D^m\phi_\d(y_n))=\td H^m(\td y_n, \pd^m\phi(\td y_n))\le 0.
\end{equation}
for  $\td y_n=\im_x(y_n)$.
\par
Now sending $n\to \infty$ and recalling \eqref{H4} and \eqref{E:4bis}, by \eqref{E:5} and  \eqref{E:6} we get \eqref{E:1}.
\end{proof}
\begin{theorem}\label{Perron}
Assume that there is a viscosity subsolution $w\in USC(\bar \O)$ and a viscosity supersolution $W\in LSC(\bar \O)$ of \eqref{HJ} such that $w\leq W$ and
\begin{equation}\label{E:7}
    w_*(x)=W^*(x)=g(x)\quad\text{for $x\in\pd \O$}.
\end{equation}
Let the function $u:\O\to\R$ be defined by $u(x):=\sup_{v\in X} v(x)$ where
\[X=\{v\in USC(\bar \O):\, \text{$v$ is a viscosity subsolution of \eqref{HJ} with $w\le v\le W$ on $\O$}\}.\]
Then, $u^*$ and $u_*$ are respectively  a subsolution and a supersolution to problem~\eqref{HJ} in $\O$ with $u=g$ on $\pd \O$.
\end{theorem}
\begin{proof}
We first show that $u$ is a subsolution of \eqref{HJ} with $u=g$ on $\pd \O$. Observe that
\begin{equation}\label{E:8}
    g(x)=w_*(x)\le u_*(x)\le u(x)\le u^*(x)\le W^*(x)=g(x)\qquad x\in\pd\O.
\end{equation}
Hence $u=u_*=u^*=g$ on $\pd \O$. Moreover, by lemma \ref{Perronsub}, $u^*$ is a viscosity subsolution of \eqref{HJ} in $\O$.
%
Let us prove by contradiction that $u_*$ is a viscosity supersolution of  \eqref{HJ} in $\O$: we assume that $u_*$ does not satisfy the  supersolution condition at some point $x_0\in\O$. We only consider the case $x_0\in\Sigma$, since the case $x_0\not\in \Sigma$ is based upon similar, but easier, arguments. We assume that there exist a vector $p_0\in T_{x_0}\Sigma$,  an index $j\in \Inc_{x_0}$
such that for any $k\in K:=\Inc_{x_0}\setminus \{j\}$,   there exists a $(j,k)$-lower test function $\phi_k$ of $u_*$ at $x_0$
with $\pi_j(D\phi_k(x_0))=p_0$ and
\begin{equation}\label{E:9}
   H^j(x_0, D^j\phi_k(x_0))=H^k(x_0,D^k\phi_k(x_0))<0.
\end{equation}
By adding a quadratic function as in the proof of lemma \ref{Perronsub}, it is not restrictive to assume that $x_0$ is a strict
minimum point for $u_*-\phi_k$, for $k\in K$.
Hence we can assume that $u_*(x_0)=\phi_k(x_0)$ and
there exists $\eta>0$ such that
\begin{equation}\label{E:8bis}
    u_*(x)-\phi_k(x)>0 \quad\text{for all $k\in K$ and $x\in \pd B_\eta(x_0)\cap(\tilde\O_j\cup \tilde\O_k) $}.
\end{equation}
where $B_\eta(x_0)=\{y\in\O:d(x_0,y)<\eta\}$. Moreover, by the continuity of $H$ and $D\phi$,
\begin{equation}\label{E:9bis}
\begin{split}
    H^j(x, D^j\phi_k(x))<0,\qquad &\forall x\in B_j:=B_\eta(x_0)\cap \td\O_j,\\
    H^k(x, D^k\phi_k(x))<0,\qquad &\forall x\in B_k:=B_\eta(x_0)\cap \td\O_k,\, k\in K.
\end{split}
\end{equation}
 Define
\[\tilde v(z):=\left\{
                  \begin{array}{ll}
                    \max_{k\in K}  \phi_k(z), & \hbox{if $z\in \O_j$,} \\
                      \phi_k(z), & \hbox{if $z\in\O_k$.}
                  \end{array}
                \right.
\]
Observe that since $\phi_k$ are continuous on $\Sigma$ with $\phi_k(x_0)=u_*(x_0)$ and $\pi_j(D^j\phi_k(x_0))=p_0$ for any $k\in K$, then
the function $\tilde v$ is continuous $B_\eta(x_0)$.
We claim that $\tilde v$ is a  subsolution of \eqref{HJ} in $B_\eta(x_0)$.
\par
\emph{Case 1:} Consider  $x\in B_\eta(x_0)\cap \Sigma$. If
  $\psi$ is  a $(l,m)$-upper test function to $\td v$ at $x$ for $l,m\in \Inc_x$, $l\neq m$, we have to show
  that
  \begin{equation}\label{E:14bis}
    H^l(x,D^l\psi(x))=\td H^l(\td x,\pd^l\psi(\td x))\le 0
  \end{equation}
where $\td x=\im_{x_0}(x)$.
Consider first the case $l, m\neq j$. By  the definition of $\td v$   it follows that
\[
 \pd_i\psi(x)=\pd_i \phi_m(x)\quad\text{for all $i=,1\dots,n-1$ and}\quad \pd_{\nu_m} \psi(x) \ge \pd_{\nu_m} \phi_m(x)
\]
Hence\begin{equation}\label{E:15}
|\pd_{\nu_m} \psi(x)|\le |\pd_{\nu_m} \phi_m(x)|\quad\text{provided that $\pd_{\nu_m} \psi(x)\le 0$}.
\end{equation}
Applying a similar argument we get that
\begin{equation}\label{E:16}
|\pd_{\nu_l} \psi(x)|\le |\pd_{\nu_l} \phi_m(x)|\quad\text{provided that $\pd_{\nu_l} \psi(x)\le 0$}.
\end{equation}
By \eqref{E:15}, \eqref{E:16} and since $\psi$ is $(l,m)$-differentiable at $x$, we have
\begin{equation}\label{E:17}
|\pd_{\nu_l} \psi(x)|=|\pd_{\nu_m} \psi(x)|\le \max\{|\pd_{\nu_l} \phi_l(x)|, |\pd_{\nu_m} \phi_m(x)|\}.
\end{equation}
By the assumptions~\eqref{H15}-\eqref{H5}, for  $s\in \Inc_x$, the function $h:\R\to\R$ defined by
\[p_n\mapsto \td h(p_n):=H^s(\td x,(\pd_1 \psi(x),\dots,\pd_{n-1} \psi(x),p_n))\]
is independent of $s$, symmetric
at $p_n=0$, and increasing in $|p_n|$. Hence by \eqref{E:9bis}, \eqref{E:17}
\begin{align*}
H^l(x,D^l \psi (x))=\td H^l(\td x,\pd^l \psi (x))=h(\partial_{\nu_l} \psi (x))\le \max\{h(\pd_{\nu_l}\phi_l(x)),\,h(\pd_{\nu_m} \phi_m(x))\}=\\
\max\{\td H^l(x, \pd^l \phi_l(x)),\,\td H^m(x,\pd^m\tilde\phi_m(x))\}=\max\{ H^l(x, D^l \phi_l(x)),\, H^m(x,D^m\phi_m(x))\}<0
\end{align*}
and we obtain \eqref{E:14bis}.\\
We now consider the case where one of the indices $l,m$ coincides with $j$, i.e. $\psi$ is a $(j,m)$-upper test function of $\td v$ at $x$.
As above we get that
 \begin{align*}
 \pd_i\psi(x)=\pd_i \phi_k(x)\quad\text{for  $i=,1\dots,n-1$, $k\in K$}\\ \pd_{\nu_j} \psi(x) \ge
 \max\{\max_{k\in K}\pd_{\nu_j} \phi_k(x),\pd_{\nu_m} \phi_m(x) \}
 \end{align*}
and therefore
\[|\pd_{\nu_j} \psi(x)|=|\pd_{\nu_m} \psi(x)|\le \max\{\max_{k\in K}|\pd_{\nu_j} \phi_k(x)|, |\pd_{\nu_m} \phi_m(x)|\}.\]
Hence we proceed as above to show the subsolution condition \eqref{E:14bis}.

\emph{Case 2:} Let $x\in B_\eta(x_0)\setminus\Sigma$. If $x\in\O_l$, $l\neq j$, then $\td v(x)=\phi_l(x)$. Since $\phi_l$ is differentiable at $x$,  we get immediately
the subsolution condition by \eqref{E:9bis}. If $l=j$, then by \eqref{E:9bis}, we have
\[H^k(x, D^k\phi_k(x))<0,\quad\text{for any $k\in K$}\]
and the claim follows by lemma \ref{Perronsub}.\\
We set $\td v_\e=\td v+\e$. Then $\td v_\e$ is still a subsolution and by \eqref{E:8bis}, for $\e$ sufficiently small,
$\td v_\e(x)<u_*(x)$ for $x\in\pd B_\eta(x_0)$. Since $u_*\le u$ on $\O$ we have $\td v_\e <u$ on $\pd B_\eta (x_0)$.
Hence the function $v:\O\to\R$ given by
\[
    v(x)=\left\{
           \begin{array}{ll}
             \max\{u(x),\td v_\e(x)\}, & \hbox{if $x\in B_\eta (x_0)$} \\
             u(x), & \hbox{if $x\not\in B_\eta (x_0)$}
           \end{array}
         \right.
\]
is upper semi-continuous in $\O$ and, by lemma \ref{Perronsub}, it is a subsolution of \eqref{HJ}. It follows that $v\in X$. Since
$v(x_0)=u_*(x_0)+\e$, there is $\bar x\in B_\eta(x_0)$ such that $v(\bar x)>u(\bar x)$ and we get a contradiction to the definition of $u$.
\end{proof}
The proof of the next   proposition  is based on  arguments similar to those used for lemma~\ref{Perronsub} and theorem \ref{Perron}
(see also \cite[prop.3.5]{sc})
\begin{proposition}\label{Perronbis}
Let $\cV$ be an arbitrary set of solutions of \eqref{HJ}. Define  $u(x):=\inf_{v\in \cV} v(x)$ for all $x\in\O$ and assume that
$u(x)\in \R$ for some  $x\in\Omega$. Then $u$ is a solution  of \eqref{HJ}.
\end{proposition}
%
\section{A representation formula for the solution of the Dirichlet problem}\label{section6}
In this section, we shall assume that, beside hypotheses \eqref{H15}-\eqref{H5},  $H$ also fulfills
\begin{equation}\label{H6}
H^j(x,p) \text{ is strictly convex in $p\in\R^n$, for $x\in\tilde \O_j$}.
\end{equation}
Observe that, under assumptions~\eqref{H6} and~\eqref{H5}, hypothesis~\eqref{H2} becomes redundant (see also remark~\ref{comm_on_hyp} for others comments on the hypotheses).
Our aim is to introduce a distance-like function and to give a representation formula for the solution of the Dirichlet problem.
\begin{definition}\label{path}
Let $x,y\in \bar\O$ and $T>0$. A continuous curve $\g:[0,T]\to \bar\O$ is called a connection of $x$ and $y$ if
\begin{itemize}
   \item $\g(0)=x$, $\g(T)=y$;
   \item $\g([0,T])\subset\bar\O$;
   \item $\g$ is an absolutely continuous path in the sense that there are
$t_0:=0<t_1<\dots<t_{M+1}:=T$ such that for any $m=0,\dots, M$, we have $\g([t_m,t_{m+1}])\subset \bar \O_{j_m}\subset\R^n$ for some
$j_m\in J$ and the curve $\g_m: [t_m,t_{m+1}]\to\bar \O_{j_m}$ defined by $\g_m(t):=\gamma(t)$ is absolutely continuous. We assume that the number $t_m$ is maximal with the property that $\g([t_m,t_{m+1}])\subset \bar \O_{j_m}$, $m=1,\dots,M$.
\end{itemize}
We denote by $\cB^T_{x,y}$ the set of all connections of $x$ and $y$. For $\gamma\in\cB^T_{x,y}$, we define the length by
$\cL(\gamma):=\sum_{l=1}^M \cL(\g_l)$
where $\cL(\g_l)$ is the usual length in $\R^n$.
\end{definition}
For $x,y\in\O$, we define
\begin{equation}\label{D:1}
    S(x,y)=\inf\left\{\int_0^T L (\g(s),\dot\g(s))ds:\,T
    >0,\,\g\in \cB^T_{x,y}\right\},
\end{equation}
where $\cB^T_{x,y}$ as in definition \ref{path} and $L=(L^j)_{j\in J}$ is a family of mappings $L^j:\tilde\O_j\times T_x^j\O\to \R$ defined by
\[L^j(x,q):=\sup_{p\in\R}\{p\cdot q- H^j(x,p)\}\qquad x\in\O_j,\,q\in\R^n.\]
\begin{proposition}\label{distance}
Assume that there exists a subsolution of \eqref{HJ}. Then\par
(i)  For any $x,y\in\O$, $S(x,y)$ is finite, Lipschitz continuous in its variables, $S(x,x)=0$   and satisfies
the triangular inequality. Moreover for any subsolution $u$ of \eqref{HJ}
\begin{equation}\label{D:1bis}
u(x)-u(y)\le S(y,x)\qquad \text{for any $x$, $y\in\O$.}
\end{equation}\indent
(ii)  For any fixed $x\in\O$, $u(\cdot)=S(x,\cdot)$ is a subsolution of \eqref{HJ}
in $\O$ and a supersolution in $\O\setminus\{x\}$.
\end{proposition}
\begin{proof}
The proof of (i) is standard and for details we refer to \cite{i}.
For (ii), we first prove that $u$ is a subsolution in $\O$. We consider $y\in  \O$ and we distinguish two cases.\par
\emph{Case 1: $y\in\O_j$ for some $j\in J$.}
We recall that, near $y$, $\O$ is locally  homeomorphic to a $n$-dimensional Euclidean space via the canonical identification chart
$(V_{y},\im_{y})$  (see prop. \ref{canonical}). Recall also that $\im_{y}(y)=0$. For $q\in\R^n$, we consider $t>0$ sufficiently small in such a way that $\td z_t:=tq\in \im_y(V_y)$ and set $z_t=\im_y^{-1}(tq)\in V_y$.
Define a trajectory $\g\in\cB_{z_t,y}^t$  by $\g_t(s)=\im_{y}^{-1}((1-s/t)\td z_t)$, $s\in [0,t]$.
Let $\psi$ be an upper test function for $u(\cdot)=S(x,\cdot)$ at $y$, hence
$u(y)-\psi(y)\ge u(z)-\psi(z)$ for $z\in V_y$.
We want to show that
\begin{equation}\label{D:2}
    H^j(y,D^j\psi(y))\le 0.
\end{equation}
We have
\begin{align*}
& \pd^j \psi(0)\cdot q=   \pd^{j}(\psi\circ \im_y^{-1})(0)\cdot q=\lim_{t\to 0^+}\frac{\psi (\im_y^{-1}(\td z_t))-\psi(\im_y^{-1}(0))}{t}\\&\le \lim_{t\to 0^+}\frac{S(x,z_t)-S(x,y)}{t} \le \lim_{t\to 0^+}\frac{S(z_t,y)}{t}\le \lim_{t\to 0^+}\frac{1}{t}\int_0^t L^j(\g_t(s),\dot\g_t(s))ds=\\
&  \lim_{t\to 0^+}\frac{1}{t}\int_0^t L^j\big(\im_y^{-1}((1-s/t)\td z_t),q\big)ds
=\td L(0,q)
\end{align*}
where  $\td L^j(  y,q)=L^j(\im_y^{-1}(y),q)$ is the dual function of the Hamiltonian $\td H^j( y,p)$ defined in \eqref{Hcanonical}. By the
previous inequality and the arbitrariness of $q$
we get $\td H^j( 0, \pd^j \psi(0))=\sup_{q\in\R^n}\{\pd^j\psi(0)\cdot q-\td L^j(0,q)\}\le 0$ and therefore \eqref{D:2} thanks to \eqref{Hcanonical}.\par
\emph{Case 2: $y\in\Sigma$.}
In this case, near $y$, $\O$ is locally homeomorphic to an elementary ramified space  $\cR^n_{r(y)}$. As before we consider the canonical immersion
$(V_{y},\im_{y})$. Let $\psi$ be an $(j,k)$-upper test function to $u(\cdot)=S(x,\cdot)$ at $y$.
For $q\in \im_y(\tilde \O_j\cup\tilde \O_k)$, we consider $z_t$ and $\tilde z_t$ as before and we introduce the trajectory $\g\in\cB_{z_t,y}^t$ by $\g_t(s)=\im_{y}^{-1}((1-s/t)\td z_t)$, $s\in [0,t]$. We have two cases
\begin{align}
    &\text{$\g_t(s)\in\Sigma\cap V_y$ for any $s\in [0,t]$ }\label{D:3}\\
    &\text{$\g_t(s)\in\O_m\cap V_y$ for any $s\in [0,t]$ for either $m=j$ or $m=k$.}\label{D:4}
\end{align}
If \eqref{D:3} is satisfied, since $\g_t\in\td\O_j\cap\td\O_k$, arguing as above we get  that
\begin{equation}\label{D:5}
\pd^m\psi(0)\cdot q-\td L^m(0,q)\le 0\qquad   \text{for $m=k, j$.}
\end{equation}
In other case,  we get
\[
\pd^m\psi(0)\cdot q-\td L^m(0,q)\le 0 \qquad \text{ for either $m=j$ or $m=k$.}
\]
If $m=k$, since $\psi$ is differentiable and \eqref{H4} holds, we have
\begin{equation}\label{D:6}
   \pd^k\psi(0)\cdot q-\td L^k(0,q)\le 0= (\pd_1\psi(0),\dots,\pd_{n-1}\psi(0),-\pd_{\nu_j}\psi(0))\cdot q-\td L^j(0,q)\le 0
\end{equation}
Recalling \eqref{H5}, by \eqref{D:5} and \eqref{D:6}, we get \eqref{D:2}.\\
The proof that $u$ is a supersolution  is based on the same argument of
the proof of theorem~\ref{Perron}. In fact, by \eqref{D:1bis} and the part already proved, it follows immediately that
\begin{equation}\label{D:7}
    u(y)=\max\{v(y):\,\text{$v$ is a subsolution of \eqref{HJ} s.t. $v(x)=0$}\}.
\end{equation}
If we assume by contradiction that $u$ is not a supersolution at some point $z\in\O\setminus\{y\}$,
we get a contradiction to \eqref{D:7} showing  as in the proof of theorem \ref{Perron} that there exists a subsolution $v$ of \eqref{HJ} such that $v(z)>u(z)$.
\end{proof}
In the following theorem, we give a representation formula for the solution of the Dirichlet problem for \eqref{HJ}.
\begin{theorem}
Assume that
\begin{equation}\label{H7}
\begin{split}
 &\text{there exists a differentiable function $\psi$ and a continuous function  $h$   }\\
    &\text{with $h<0$  such that  $H(x,D \psi)\le h(x)$ for $x\in\O$.}
    \end{split}
\end{equation}
Consider the Hamilton-Jacobi equation \eqref{HJ} with the boundary condition
\begin{equation}\label{BC}
    u(x)=g(x),\qquad x\in \partial\O
\end{equation}
where $g:\pd \O\to \R$ is a continuous function satisfying
\begin{equation}\label{H8}
    g(x)-g(y)\le S(y,x)\qquad\text{for any $x$, $y\in  \partial \O$.}
\end{equation}
Then the unique viscosity solution of \eqref{HJ}--\eqref{BC} is given by
\[u(x):=\min\{g(y)+S(y,x):\, y\in  \partial \O\}.\]
\end{theorem}
\begin{proof}
First observe that if   $u(x)\neq g(x)$ for  some $x\in   \partial \O$, then there  is  $z\in   \partial \O$
such that $g(x)>S(z,x)+g(z)$, which gives  a contradiction to \eqref{H8}. Hence $u=g$ on $\pd \O$.\\
By proposition \ref{Perronbis}   and  proposition \ref{distance}.(ii), $u$ is a solution of \eqref{HJ}. To show that it is the unique solution, assume that there exists another solution $v$ of \eqref{HJ}-\eqref{BC}. For $\theta\in (0,1)$ define
$u_\theta :=\theta u+(1-\theta)\psi$, where $\psi$ as in \eqref{H7}. By adding a constant it is not restrictive to assume that $\psi$ is sufficiently small in such a way that
 \begin{equation}\label{4:2}
 u_\theta(x)\le u(x),\qquad x\in \O.
 \end{equation}
If $x\in  \O_j$ and  $\phi$ is  an upper test
function of $u$ at $x$, we set  $\phi_\theta:=\theta\phi +(1-\theta)\psi$  and we obtain by means of convexity of $H$
 \begin{equation}\label{4:3}
    H^j(x, D^j \phi_\theta)\le \theta H^j(x,D^j\phi)+(1-\theta)H^j(x,D^j\psi)\le (1-\theta)h(x).
\end{equation}
If  $x\in \Sigma$  and
$\phi$  is an  $(j,k)$-upper test function of $u$ at $x$, then   $\phi_\theta:=\theta\phi +(1-\theta)\psi$  is a  $(j,k)$-upper test function of $u_\theta$ at $x$ and we again obtain \eqref{4:3}. Hence $u_\theta$ is a viscosity subsolution
of
\[ H(x, D^j u)\le (1-\theta)h(x).\]
Applying theorem \ref{Comparison} with $f=(1-\theta)h$ and \eqref{4:2}, we get  $u_\theta\le v$ for all $\theta\in (0,1)$. Letting $\theta$ tend to $1$ yields
$u\le v$. Exchanging the role of $u$ and $v$ we conclude that $u=v$ in $\O$.
\end{proof}
\begin{remark}
The assumption \eqref{H7} is satisfied if $H(x,0)<0$ for any $x\in \O$ by taking   $\psi\equiv c\in\R$. If $f>0$  in $\O$  then the Hamiltonian
$H=(H_j)_{j\in J}$, where $H^j(x,p)=|p|^2-f^j(x)$ (see example \ref{eikonal}),  satisfies this assumption.
\end{remark}
\bibliographystyle{amsplain}

\end{document}